\documentclass[11pt, a4paper]{amsart}

\usepackage[english]{babel}
\usepackage{amsmath,amssymb}

\newtheorem{thmintro}{Theorem}

\newtheorem{corintro}[thmintro]{Corollary}
\newtheorem{theorem}{Theorem}[section]
\newtheorem{corollary}[theorem]{Corollary}
\newtheorem{lemma}[theorem]{Lemma}
\newtheorem{prop}[theorem]{Proposition}
\theoremstyle{definition}
\newtheorem{remarkintro}{Remark}
\newtheorem{exampleintro}{Example}
\newtheorem{remark}[theorem]{Remark}

\newcommand{\NN}{\mathbb{N}}
\newcommand{\RR}{\mathbb{R}}
\newcommand{\ZZ}{\mathbb{Z}}

\newcommand{\CCCC}{\mathcal{C}}

\newcommand{\inv}{^{-1}}
\newcommand{\la}{\langle}
\newcommand{\ra}{\rangle}
\newcommand{\co}{\colon\thinspace}

\DeclareMathOperator{\CAT}{CAT(0)}
\DeclareMathOperator{\Cent}{(Cent)}
\DeclareMathOperator{\dist}{d}
\DeclareMathOperator{\interieur}{int}
\DeclareMathOperator{\Min}{Min}
\DeclareMathOperator{\Pc}{Pc}

\begin{document}

\renewcommand{\proofname}{{\bf Proof}}

\title[Conjugacy classes and straight elements in Coxeter groups]{Conjugacy classes and straight elements in Coxeter groups}

\author[Timothée~Marquis]{Timoth\'ee \textsc{Marquis}$^*$}
\address{UCL, 1348 Louvain-la-Neuve, Belgium}
\email{timothee.marquis@uclouvain.be}
\thanks{$^*$F.R.S.-FNRS Research Fellow}

%\date{October 2013}
 % AMS classification numbers

\begin{abstract}
Let $W$ be a Coxeter group. In this paper, we establish that, up to going to some finite index normal subgroup $W_0$ of $W$, any two cyclically reduced expressions of conjugate elements of $W_0$ only differ by a sequence of braid relations and cyclic shifts. This thus provides a very simple description of conjugacy classes in $W_0$. As a byproduct of our methods, we also obtain a characterisation of straight elements of $W$, namely of those elements $w\in W$ for which $\ell(w^n)=n\thinspace\ell(w)$ for any $n\in\ZZ$. In particular, we generalise previous characterisations of straight elements within the class of so-called cyclically fully commutative (CFC) elements, and we give a shorter and more transparent proof that Coxeter elements are straight.
\end{abstract}

\maketitle

\section{Introduction}

Let $(W,S)$ be a Coxeter system. By a classical result of J.~Tits \cite{Titswordproblem} (also known as Matsumoto's theorem, see \cite{Matsumoto}), any two reduced expressions of a given element of $W$ only differ by a sequence of braid relations. This yields in particular a very simple solution to the word problem for Coxeter groups. 

In attempting to find an analoguous solution for the conjugacy problem (see \cite{Krammer} for a thorough study of this problem), one might investigate a cyclic version of this theorem, and ask \emph{whether two cyclically reduced expressions of conjugate elements of $W$ only differ by a sequence of braid relations and cyclic shifts} (see \cite{CFCCoxeter}). We recall that $w\in W$ is {\bf cyclically reduced} if every cyclic shift of any reduced expression for $w$ is still reduced. 

Although the answer to this question is ``no" in general (see Remark~\ref{remarkintro 2} below), it is probably true within the class of elements of $W$ whose parabolic closure has only infinite irreducible components (more generally, it might be true that two cyclically reduced conjugate elements can be obtained from one another by a sequence of braid relations and cyclic shifts if and only if they have the same parabolic closure). Up to now, this ``cyclic version" of Matsumoto's theorem had only be shown to hold within the class of Coxeter elements of $W$ (see \cite{MR2515767}). 
Our first theorem is a dramatic strengthening of this result. 

We say that an element $u\in W$ has property {\bf (Cent)} if whenever $u$ normalises a finite parabolic subgroup of $W$, it centralises this subgroup.

\begin{thmintro}\label{thmintro Matsumoto}
Let $u$ be a cyclically reduced element of $W$ of infinite order. Assume that $u$ has property $\Cent$. Then any cyclically reduced conjugate of $u$ can be obtained from $u$ by a sequence of braid relations and cyclic shifts.
\end{thmintro}

This thus provides a very simple description of conjugacy classes of elements of infinite order of $W$ with property (Cent). Given a word $\omega=s_1\dots s_js_{j+1}\dots s_k$ with each $s_i\in S$ and such that $s_j=s_{j+1}$, we say that the word $\omega'=s_1\dots s_{j-1}s_{j+2}\dots s_k$ is obtained from $\omega$ by {\bf ss-cancellation}.

\begin{corintro}\label{corollary conjugacy classes}
Let $u_1$, $u_2$ be two elements of infinite order of $W$ and assume that $u_1$ has property $\Cent$. Then $u_1$ and $u_2$ are conjugate if and only if there exists some $w\in W$ obtained from each of $u_1$ and $u_2$ by a sequence of braid relations, cyclic shifts, and ss-cancellations.  
\end{corintro}

The proof of these results is given in Section~\ref{section Matsumoto}.

\begin{remarkintro}\label{remarkintro 1}
Note that the class of elements of infinite order with property $\Cent$ is very large. For example, it contains any torsion-free finite index normal subgroup $W_0$ of $W$ (see \cite[Lemma~1]{DraJa}). Note that the existence of such a subgroup $W_0$ is ensured by Selberg's lemma.
\end{remarkintro}

\begin{remarkintro}\label{remarkintro 1.2}
One can in fact give a slightly more precise version of property $\Cent$ under which the conclusion of Theorem~\ref{thmintro Matsumoto} still holds, by specifying the spherical parabolic subgroups that should be centralised (see Remark~\ref{remark precision Cent}). This is property {\bf (Cent')}, which is by definition satisfied by an element $u\in W$ if every conjugate $w$ of $u$ obtained from $u$ by a sequence of braid relations and cyclic shifts has the following property: whenever $w$ normalises a spherical parabolic subgroup of the form $w_IW_Jw_I\inv$ for some spherical subsets $J\subseteq I\subseteq S$ and some element $w_I\in W_I$, it centralises $w_IW_Jw_I\inv$.
\end{remarkintro}

\begin{remarkintro}\label{remarkintro 2}
Assume that $W$ possesses two conjugate standard parabolic subgroups, say $W_T=\la T\ra$ and $W_U=\la U\ra$ for some distinct subsets $T$ and $U$ of $S$. Then by \cite[Proposition~3.1.6]{Krammer} these subsets are conjugate, say $xTx\inv=U$ for some $x\in W$. Then for any cyclically reduced element $w$ whose (standard) parabolic closure is $W_T$, the element $xwx\inv$ is a cyclically reduced conjugate of $w$ that cannot be obtained from $w$ by a sequence of braid relations and cyclic shifts. Thus, if one wants to describe the conjugacy classes of arbitrary elements of $W$, one should allow, besides braid relations and cyclic shifts, conjugations by elements that conjugate a subset of $S$ to another one.

Note that the conjugate standard parabolic subgroups of $W$ were completely described by Deodhar (see e.g. \cite[Theorem~3.1.3]{Krammer}) and that this phenomenon is purely spherical: if $W_T$ is irreducible and conjugate to $W_U$ for some distinct subsets $T$ and $U$ of $S$, then $T$ is spherical.
\end{remarkintro}

Finally, we mention that a weaker version of Theorem~\ref{thmintro Matsumoto} for torsion elements of $W$ has already been studied in several papers, in connection with characters of Hecke algebras (see \cite[Theorem~2.6]{GeckKimPfeiffer} and the references in that paper). More precisely, it is known that, given an element $w$ in a finite Coxeter group $W$, one can obtain, using only braid relations, cyclic shifts and ss-cancellations, a conjugate $u$ of $w$ of minimal length in its conjugacy class. Moreover, for any two conjugate elements $u,u'\in W$ of minimal length in their conjugacy class, there is a sequence $u=u_0,u_1,\dots,u_k=u'$ and elements $x_1,\dots,x_k\in W$ such that $u_{i-1}\stackrel{x_i}{\sim}u_{i}$ for each $i=1,\dots,k$, where for any $w,w',x\in W$, we write $w \stackrel{x}{\sim}w'$ if $w'=x\inv wx$ and if $\ell(x\inv w)=\ell(x)+\ell(w)$ or $\ell(wx)=\ell(x)+\ell(w)$. As a corollary of the first statement and of our methods, we deduce in particular the following non-obvious fact, whose proof can be found at the end of Section~\ref{section Matsumoto}.

\begin{corintro}\label{corintro equiv cycl red}
Assume that $w\in W$ either has property $\Cent$ or has finite order. Then $w$ is cyclically reduced if and only if it is of minimal length in its conjugacy class.
\end{corintro}

\noindent
A second related problem, which we investigate in this paper as well, is to characterise the straight elements of $W$ where, following \cite{Krammer}, we call an element $w\in W$ {\bf straight} if $\ell(w^n)=n\thinspace\ell(w)$ for all $n\in\NN$.

The question whether Coxeter elements are straight has been the object of several papers (see \cite{wkreduced} and references therein) and this question has only been settled in full generality in 2009 by D.~Speyer (see \emph{loc.cit.}). 

More generally, one might try to find natural necessary and sufficient conditions for an arbitrary element $w\in W$ to be straight. This problem for instance motivated the paper \cite{CFCCoxeter}, in which the authors characterise the straight elements in the class of so-called \emph{cyclically fully commutative} (CFC) elements (which ``generalises" in some sense the class of Coxeter elements), provided that $W$ satisfies some additional technical condition.

An obvious requirement for an arbitrary element $w\in W$ to be straight is that it should be cyclically reduced (see e.g. Lemma~\ref{lemma log implies} below). This is one of the two conditions considered in \cite{CFCCoxeter}; the second condition is what the authors of \emph{loc.cit.} call \emph{torsion-freeness}, namely, the fact for an element $w\in W$ to have a standard parabolic closure with only infinite irreducible components. This second condition is however too restrictive in general. In this paper, we give the following less restrictive definition of torsion-freeness (which coincides with the one given in \emph{loc.cit.} within the class of CFC elements, see Lemma~\ref{lemma FC tf}): an element $w\in W$ is {\bf torsion-free} if $w$ has no reduced decomposition of the form $w=w_In_I$ for some spherical subset $I\subseteq S$, some $w_I\in W_I\setminus\{1\}$ and some $n_I$ normalising $W_I$. This is indeed an obvious requirement for $w$ to be straight (see Lemma~\ref{lemma log implies}). 

Our second theorem states that these two conditions are the only possible obstructions to straightness, without any restriction on $W$ or on the class of elements of $W$ considered. %Note that a cyclically reduced element of $W$ is straight if and only if all its cyclically reduced conjugates are straight (see Lemma~\ref{lemma log iff conj log}).

\begin{thmintro}\label{mainthm intro}
Let $(W,S)$ be a Coxeter system and let $u\in W$ be cyclically reduced. Then the following are equivalent:
\begin{itemize}
\item[(1)] $u$ is straight.
%\item[(2)] Every cyclically reduced conjugate of $u$ is torsion-free.
\item[(2)] Every cyclically reduced conjugate of $u$ obtained from $u$ by a sequence of braid relations and cyclic shifts is torsion-free.
\end{itemize}
\end{thmintro}

\begin{exampleintro}
Let $W$ be the affine Coxeter group of type $\tilde{A_2}$, with set of generators $S=\{s,t,u\}$. Thus $(st)^3=(tu)^3=(su)^3=1$. Consider the cyclically reduced element $w\in W$ with reduced expression $w=tustuts$. Note then that $w^2=(tus)^4$ and hence $w$ is not straight. Note also that $w$ is torsion-free as it does not normalise any nontrivial finite standard parabolic subgroup of $W$. However, its cyclic shift $sws=(stustu)t=t(stustu)$ is not torsion-free (with $I=\{t\}$). 
\end{exampleintro}

As a corollary, we generalise the characterisation of straight CFC elements given in \cite[Corollary~7.2]{CFCCoxeter} by removing the technical assumption on $W$ considered in that paper. We recall that an element $w\in W$ is {\bf fully commutative} (FC) if any two reduced expressions for $w$ can be obtained from one another by iterated commutations of commuting generators. It is moreover {\bf cyclically fully commutative} (CFC) if every cyclic shift of any reduced expression of $w$ is a reduced expression for a FC element. We also recall that the {\bf standard parabolic closure} of an element $w\in W$ is the smallest standard parabolic subgroup of $W$ containing $w$, or else the subgroup of $W$ generated by the elements of $S$ appearing in any reduced decomposition of $w$.

\begin{corintro}\label{corintro CFC}
Let $w$ be a CFC element of $W$. Then the following are equivalent:
\begin{itemize}
\item[(1)] $w$ is straight.
\item[(2)] The standard parabolic closure of $w$ has only infinite irreducible components.
\end{itemize}
\end{corintro}

In particular, this yields a shorter and more transparent proof that Coxeter elements are straight.
\begin{corintro}\label{corintro Coxeter element}
Let $u=s_1\dots s_n$ be a Coxeter element in $W$. Then $u$ is straight if and only if $W$ has only infinite irreducible components.
\end{corintro}

The proof of these results is given in Section~\ref{section proof of thm}.

\section{Preliminaries}

\subsection{Basic definitions}
Basics on Coxeter groups and complexes can be found in \cite[Chapters~1--3]{BrownAbr}.

Let $(W,S)$ be a Coxeter system with associated Coxeter complex $\Sigma=\Sigma(W,S)$ and with set of roots (or half-spaces) $\Phi$. We denote by $C_0=\{1_W\}$ the fundamental chamber of $\Sigma$ and by $\{\alpha_s | s\in S\}$ the set of simple roots. We write $\ell(w)$ for the length of an element $w\in W$, that is, the number of generators (from $S$) appearing in a reduced decomposition of $w$. 

We recall that a {\bf gallery} $\Gamma$ between two chambers $D$ and $E$ of $\Sigma$ is a sequence $D=D_0,D_1,\dots,D_k=E$ of chambers such that for each $i=1,\dots,k$, the chambers $D_{i-1}$ and $D_{i}$ are $s_i$-adjacent for some $s_i\in S$, where adjacency is understood as adjacency in the Cayley graph of $(W,S)$. The integer $k\geq 1$ is called the {\bf length} of $\Gamma$, and the sequence $(s_1,s_2,\dots,s_k)\in S^k$ is its {\bf type}. A {\bf minimal gallery} between the chambers $C$ and $D$ is then a gallery between $C$ and $D$ of minimal length. Given two chambers $vC_0$ and $wvC_0$ ($v,w\in W$), the function associating to a gallery its type establishes a bijective correspondance between minimal galleries from $vC_0$ to $wvC_0$ and reduced expressions for $v\inv wv\in W$. In particular, the length $\ell(w)$ of an element $w\in W$ is also the length of a minimal gallery from $C_0$ to $wC_0$, or else the number of walls crossed by such a gallery, that is, the number of walls separating $C_0$ from $wC_0$.

For a subset $I$ of $S$, we let $W_I$ denote the corresponding {\bf standard parabolic subgroup}, that is, the subgroup of $W$ generated by $I$. Conjugates in $W$ of standard parabolic subgroups are called {\bf parabolic subgroups} of $W$. It is a standard fact that any intersection of parabolic subgroups of $W$ is again a parabolic subgroup, and it thus makes sense to define the {\bf parabolic closure} of a subset $H$ of $W$, which we denote by $\Pc(H)$. In case $H$ is a singleton $\{w\}$, we will also write $\Pc(w):=\Pc(\{w\})$. 

For a subset $I$ of $S$, we also let $\Phi_I$ denote the subset of $\Phi$ consisting of the roots of the form $v\alpha_s$ for some $v\in W_I$ and $s\in I$. Finally, we write $N_W(W_I)$ for the normaliser of $W_I$ in $W$ and we let $N_I$ denote the stabiliser in $W$ of the set of roots $\{\alpha_s | s\in I\}$. The following lemma follows from \cite[Lemma~5.2]{Lusztig}.

\begin{lemma}\label{lemma normaliser}
Let $I$ be a spherical subset of $S$. Then $N_W(W_I)=W_I\rtimes N_I$ and $\ell(w_In_I)=\ell(w_I)+\ell(n_I)$ for any $w_I\in W_I$ and $n_I\in N_I$.
\end{lemma}

\subsection{Davis complex}\label{subsection Davis complex}
Let $|\Sigma|$ be the standard geometric realisation of $\Sigma$. We denote by $X$ the Davis complex of $W$. Thus $X$ is a $\CAT$ subcomplex of the barycentric subdivision of $|\Sigma|$ on which $W$ acts by cellular isometries. Moreover, each point $x\in X$ determines a unique spherical simplex of $\Sigma$, called its {\bf support}. In this paper, we identify the walls, roots and simplices in $\Sigma$ with their (closed) realisation in $X$. More precisely, we will identify a simplex $A$ of $\Sigma$ with the set of $x\in X$ whose support is a face of $A$. In particular, we view chambers as closed subsets of $X$. 

\subsection{Actions on CAT(0)-spaces}
Consider the $W$-action on $X$. For an element $w\in W$, we let $$|w|:=\inf\{d(x,w x) \ | \ x\in X\}\in [0,\infty)$$ denote its {\bf translation length} and we set $$\Min(w):=\{x\in X \ | \ d(x,wx)=|w|\}.$$ A standard result of M.~Bridson (see \cite{Bridson}) asserts that such an action is {\bf semi-simple}, meaning that $\Min(w)$ is nonempty for any $w\in W$. If $w$ has finite order, then $w$ has a fixed point (see e.g. \cite[Theorem~11.23]{BrownAbr}). If $w$ has infinite order, then $|w|\neq 0$ and $\Min(w)$ is the union of all the $w$-axes, where a {\bf $w$-axis} is a geodesic line stabilised by $w$ (on which $w$ then acts by translation). Basics on $\CAT$ spaces may be found in \cite{BHCAT0}. 

\subsection{Walls}
Let $w\in W$ be of infinite order. Given a $w$-axis $L$, we say that a wall is \textbf{transverse} to $L$ if it intersects $L$ in a single point. We call a wall $m$ \textbf{$w$-essential} if it is transverse to some $w$-axis. 

We recall that the intersection of a wall and any geodesic segment in $X$ that is not completely contained in that wall is either empty or consists of a single point (see \cite[Lemma~3.4]{singlepoint}). Given $x, y \in X$, we say that a wall $m$ \textbf{separates} $x$ from $y$ if the intersection of $m$ with the geodesic segment joining $x$ to $y$ consists of a single point.

\section{A cyclic version of Matsumoto's theorem}\label{section Matsumoto}
This section is devoted to the proof of Theorem~\ref{thmintro Matsumoto} and of its corollaries. Throughout this section, we fix a Coxeter system $(W,S)$.

Let $w\in W$. Given two chambers $C,D$ of $\Sigma$, we denote by $\Gamma(C,D)$ their {\bf convex hull}, namely, the reunion of all minimal galleries from $C$ to $D$, viewed as a subcomplex of $\Sigma$. As mentioned in Section~\ref{subsection Davis complex}, we will also view $\Gamma(C,D)$ as a closed subcomplex of $X$. Let $\CCCC_0$ be the set of chambers of $\Sigma$ of the form $w^kC_0$ for some $k\in\ZZ$, and define inductively $$\CCCC_{i+1}:=\bigcup_{D\in\CCCC_i}{\Gamma(D,wD)}\subseteq X$$ for each $i\geq 0$. Set $\CCCC_w:=\bigcup_{i\geq 0}{\CCCC_i}\subseteq X$.

Before describing the basic properties of $\CCCC_w$, we need the following technical lemma.
\begin{lemma}\label{lemma geodesic in gallery}
Let $C$ and $D$ be two chambers of $\Sigma$, and pick points $x\in C\subset X$ and $y\in D\subset X$. Then there exists a minimal gallery from $C$ to $D$ containing the geodesic segment from $x$ to $y$.
\end{lemma}
\begin{proof}
If $x$ and $y$ lie in the interior of $C$ and $D$ respectively, then the lemma follows from \cite[Proposition~12.25]{BrownAbr}. 
In the general case, there exists a sequence $(x_n)_{n\in\NN}$ (resp. $(y_n)_{n\in\NN}$) of points in the interior of $C$ (resp. $D$) that converges to $x$ (resp. $y$). Thus for each $n\in\NN$, the geodesic segment $[x_n,y_n]$ from $x_n$ to $y_n$ is entirely contained in some minimal gallery from $C$ to $D$. As there are only finitely many such galleries, we may assume, up to choosing a subsequence, that $[x_n,y_n]$ is entirely contained in a given minimal gallery $\Gamma$ between $C$ and $D$ for all large enough $n$. As $\Gamma$ is closed, it then also contains the geodesic segment $[x,y]$ from $x$ to $y$, as desired. 
\end{proof}

\begin{lemma}\label{lemma basic prop CCCC}
We have the following:
\begin{itemize}
\item[(1)]
Each $\CCCC_i$, $i\in\NN$, is $w$-stable. In particular, $\CCCC_w$ is $w$-stable.
\item[(2)]
$\CCCC_w$ is a closed subset of $X$.
\item[(3)]
Given a point $x\in\CCCC_w\subseteq X$, the geodesic from $x$ to $wx$ is entirely contained in $\CCCC_w$. 
%\item[(4)]
%$\CCCC_w$ is contained in the intersection of all half-spaces that contain $\CCCC_0$.
\end{itemize}
\end{lemma}
\begin{proof}
The first statement follows from a straightforward induction on $i$.

To see that $\CCCC_w$ is closed, note that any sequence of points of $\CCCC_w$ that converges in $X$ is bounded, hence contained in finitely many chambers of $\CCCC_w$. As the reunion of these chambers is a closed subset of $\CCCC_w$, the conclusion follows.

We now prove (3). Let $x\in\CCCC_w\subseteq X$. Pick some $i\geq 0$ and some chamber $D\in\CCCC_i$ containing $x$. Then by Lemma~\ref{lemma geodesic in gallery}, the geodesic from $x$ to $wx$ is contained in some minimal gallery from $D$ to $wD$, and hence in $\CCCC_{i+1}\subseteq\CCCC_w$, as desired.
%Finally, as $\CCCC_i$ is in the intersection of all half-spaces that contain $\CCCC_0$ for any $i\geq 0$ (this follows from a straightforward induction on $i$), statement (4) follows. 
\end{proof}

We say that $u\in W$ is {\bf elementary related} to $v\in W$ if $v$ admits a decomposition that is a cyclic shift of some reduced decomposition of $u$, that is, there is some reduced decomposition $u=s_1\dots s_n$ ($s_i\in S$) of $u$ and some $k\in\{1,\dots,n\}$ such that $v=s_k\dots s_ns_1\dots s_{k-1}$. We say that $u\in W$ is {\bf $\kappa$-related} to $v\in W$, which we denote by $u\sim_{\kappa} v$, if there exists a sequence $u=u_0,u_1,\dots,u_k=v$ for some $k\geq 1$ such that $u_i$ is elementary related to $u_{i+1}$ for all $i=0,\dots,k-1$. Note that if $u\sim_{\kappa} v$, then $v$ is conjugate to $u$ and $\ell(u)\geq\ell(v)$. Note also that the relation $\sim_{\kappa}$ is in general not symmetric, as one might have $u\sim_{\kappa} v$ with $\ell(u)>\ell(v)$. However, within the class of cyclically reduced elements of $W$, $\sim_{\kappa}$ is an equivalence relation. 

\begin{lemma}\label{lemma geom combin CCCC}
Let $D$ be a chamber of $\CCCC_w$, say $D=vC_0$ for some $v\in W$. Then $w\sim_{\kappa} v\inv wv$.
\end{lemma}
\begin{proof}
Let $i\in\NN$ be such that $D\in\CCCC_i$. We prove the claim by induction on $i$.

If $i=0$, then $v=w^k$ for some $k\in\ZZ$ and there is nothing to prove.

Assume now that $i>0$. Let $E$ be a chamber of $\CCCC_{i-1}$ such that $D$ is on a minimal gallery $\Gamma$ from $E$ to $wE$, say $E=uC_0$ for some $u\in W$. Hence the type of $\Gamma$ is obtained from a reduced expression of $u\inv wu$. Thus $v=uu_1$ for some $u_1\in W$ appearing as an initial subword in some reduced expression of $u\inv wu$. In particular, $u\inv wu$ is elementary related to $v\inv wv$. As $w\sim_{\kappa} u\inv wu$ by induction hypothesis, the conclusion follows.
\end{proof}

\begin{prop}\label{prop CCCC cap Min}
Let $w\in W$. Then $\CCCC_w$ contains a point of $\Min(w)$. In particular, if $w$ has infinite order, then $\CCCC_w$ contains a $w$-axis.
\end{prop}
\begin{proof}
For each $u\in W$, consider the continuous function $f_u\co X\to \RR: x\mapsto\dist(x,ux)$. Let $u_1C_0,\dots,u_kC_0$ ($u_i\in W$) be the chambers at (gallery) distance at most $\ell(w)$ from $C_0$. For $i=1,\dots,k$, let $a_i\in \RR$ be the minimum of the function $f_{u_i}$ over $C_0$.

For each chamber $D$ of $\CCCC_w$, let $x_D$ be a point of $D$ where $f_w$ attains its minimum over $D$. Note then that if $D=vC_0$ for some $v\in W$, the function $f_{v\inv wv}$ attains its minimum over $C_0$ at $v\inv x_D$. Moreover, as $\ell(v\inv wv)\leq \ell(w)$ by Lemma~\ref{lemma geom combin CCCC}, there is some $i\in\{1,\dots,k\}$ such that $v\inv wv=u_i$. Then $f_w(x_D)=\dist(x_D,wx_D)=f_{u_i}(v\inv x_D)=a_i$. We have thus shown that the set $\{f_w(x_D) \ | \ \textrm{$D$ is a chamber of $\CCCC_w$}\}$ is contained in $\{a_i \ | \ i=1,\dots,k\}$ and is therefore finite. In particular, $f_w$ attains its minimum over $\CCCC_w$, say at $x\in\CCCC_w$.

We claim that $\dist(w\inv x,wx)=2\dist(x,wx)$, and hence that $x\in \Min(w)$ (see e.g. \cite[Chapter~II, Proposition~1.4~(2)]{BHCAT0}). Indeed, suppose for a contradiction that $\dist(w\inv x,wx)<2\dist(x,wx)$. Let $y$ be the midpoint of the geodesic from $w\inv x$ to $x$. Then $f_w(y)=\dist(y,wy)<\dist(y,x)+\dist(x,wy)=\dist(x,wx)=f_w(x)$. As $w\inv x$ and hence $y$ belong to $\CCCC_w$ by Lemma~\ref{lemma basic prop CCCC}~(1) and (3), we get the desired contradiction.

The second statement follows from the fact that, if $w$ has infinite order, then any $w$-axis intersecting $\CCCC_w$ is entirely contained in $\CCCC_w$ by Lemma~\ref{lemma basic prop CCCC}~(1) and (3).
\end{proof}

\begin{corollary}\label{corollary sim to C0}
Let $u\in W$ be of infinite order. Then there is some $w\in W$ with $u\sim_{\kappa} w$ such that $\Min(w)\cap C_0$ is nonempty and not contained in any $w$-essential wall.
\end{corollary}
\begin{proof}
By Proposition~\ref{prop CCCC cap Min}, we know that $\CCCC_u$ contains a $u$-axis. Let $x$ be a point on that axis that is not contained in any $u$-essential wall. Let $D$ be a chamber of $\CCCC_u$ containing $x$, say $D=vC_0$ for some $v\in W$. Then Lemma~\ref{lemma geom combin CCCC} implies that $u\sim_{\kappa} v\inv uv$. The conclusion follows with $w=v\inv u v$.
\end{proof}

%Call an element $w\in W$ {\bf $\kappa$-cyclically reduced} if $\ell(w)=\ell(u)$ for any $u\in W$ with $w\sim_{\kappa} u$. Note that a cyclically reduced element is $\kappa$-cyclically reduced; the two notions most likely coincide.

Here is a restatement of Theorem~\ref{thmintro Matsumoto}.
\begin{theorem}\label{theorem matsumoto}
Let $w\in W$ be cyclically reduced and of infinite order. Assume that $w$ has property $\Cent$. Then any cyclically reduced conjugate of $w$ is $\kappa$-related to $w$.
\end{theorem}
\begin{proof}
Let $u$ be a cyclically reduced conjugate of $w$. By Corollary~\ref{corollary sim to C0}, there is no loss of generality in assuming that $w$ and $u$ both possess an axis through $C_0$, say $x_w\in\Min(w)\cap C_0$ and $x_u\in\Min(u)\cap C_0$. Choose $v\in W$ of minimal length so that $u=v\inv wv$. Note then that $vx_u\in\Min(w)$. We claim that $vC_0$ is a chamber of $\CCCC_w$, which would imply the theorem by Lemma~\ref{lemma geom combin CCCC}.

Indeed, assume for a contradiction that $vC_0$ is not a chamber of $\CCCC_w$, and let $C_0,C_1,\dots,C_k=vC_0$ ($k\geq 1$) be a minimal gallery from $C_0$ to $vC_0$ containing the geodesic segment $[x_w,vx_u]$ between $x_w$ and $vx_u$ (such a gallery exists by Lemma~\ref{lemma geodesic in gallery}). Let $i\in\{0,\dots k-1\}$ be such that $D:=C_i$ is a chamber of $\CCCC_w$ but $E:=C_{i+1}$ is not. Let $m$ be the wall separating $D$ from $E$. Then $m$ intersects the geodesic $[x_w,vx_u]\subseteq\Min(w)$. Let $y\in (m\cap D)\cap [x_w,vx_u]\subseteq (m\cap D)\cap\Min(w)$, and let $L$ be the $w$-axis through $y$. If $L$ is transverse to $m$, then for some $\epsilon\in\{\pm 1\}$, the chambers $D$ and $w^{\epsilon}D$ lie on different sides of $m$, and hence there is a minimal gallery from $D$ to $w^{\epsilon}D$ passing through $E$, contradicting the fact that $E$ is not a chamber of $\CCCC_w$.
Thus $L$ is containd in $m$. As by assumption $w$ centralises the finite parabolic subgroup generated by the reflections whose wall contain $L$, we deduce that $w$ commutes with $r_m$. In particular, $u=(r_mv)\inv w(r_mv)$. Since $\ell(r_mv)<\ell(v)$, this contradicts the minimality assumption on $v$, as desired.
\end{proof}

\begin{remark}\label{remark precision Cent}
Note that Remark~\ref{remarkintro 1.2} from the introduction readily follows from the proof of Theorem~\ref{theorem matsumoto}. Indeed, keeping the notations of the above proof, let $v_1\in W$ be such that $D=v_1C_0$, so that $w\sim_K v_1\inv wv_1$ by Lemma~\ref{lemma geom combin CCCC}. Let also $A=v_1W_I$ be the support of $y$, for some spherical subset $I$ of $S$. Assume as in the proof that the $w$-axis $L$ through $y$ is contained in $m$. Let $P$ be the spherical parabolic subgroup which is the parabolic closure of the set of reflections whose wall contains $L$. Thus $w$ normalises $P$ and $P\leq v_1W_Iv_1\inv$. Hence $v_1\inv wv_1$ normalises $v_1\inv Pv_1\leq W_I$ and $v_1\inv Pv_1$ is of the form $w_IW_Jw_I\inv$ for some $w_I\in W_I$ and some subset $J\subseteq I$. The conclusion follows.
\end{remark}

\noindent
{\bf Proof of Corollary~\ref{corollary conjugacy classes}.}
Assume that $u_1$ and $u_2$ are conjugate. For $i=1,2$, let $w_i\in W$ be cyclically reduced and such that $u_i\sim_{\kappa} w_i$. It then follows from Theorem~\ref{theorem matsumoto} that $w_1\sim_{\kappa} w_2$. The conclusion follows.
\hspace{\fill}$\Box$

\medskip

We conclude this section by proving Corollary~\ref{corintro equiv cycl red}. To facilitate the exposition, we will call an element $u\in W$ \emph{strongly cyclically reduced in $H<W$} ({\bf SCR in $H$} for short) if $\ell(u)=\min\{\ell(v\inv uv) \ | \ v\in H\}$. Thus $u\in W$ is of minimal length in its conjugacy class if and only if it is SCR in $W$.

\begin{lemma}\label{lemma cyclically reduced inside Pc}
Let $w\in W$ be such that $\Pc(w)$ is standard. Then there exists some $v\in\Pc(w)$ such that $vwv\inv$ is SCR in $W$.
\end{lemma}
\begin{proof}
Let $T\subseteq S$ be such that $\Pc(w)=W_T$. Let also $v_1\in W$ be such that $v_1wv_1\inv$ is SCR in $W$. In particular, $\Pc(v_1wv_1\inv)$ is standard (see e.g. \cite[Proposition~4.2]{MR2585575}) and hence $v_1W_Tv_1\inv=\Pc(v_1wv_1\inv)=W_U$ for some $U\subseteq S$. It then follows from \cite[Proposition~3.1.6]{Krammer} that there is some $u\in W_U$ such that $v_1 T v_1\inv =uUu\inv$. Set $x=u\inv v_1$, so that $x\inv Ux=T$. Set also $v=x\inv ux\in W_T$. Note that $u(xwx\inv)u\inv\in W_U$ and hence $vwv\inv=x\inv u(xwx\inv)u\inv x$ has length $\ell(u(xwx\inv)u\inv)=\ell(v_1wv_1\inv)$. Thus $vwv\inv$ is SCR in $W$, as desired.
\end{proof}

\noindent
{\bf Proof of Corollary~\ref{corintro equiv cycl red}.}
Clearly, if $w$ is SCR in $W$, then it is cyclically reduced. Conversely, assume that $w$ is cyclically reduced and let us prove that it is SCR in $W$. If $w$ has property $\Cent$, this readily follows from Theorem~\ref{thmintro Matsumoto}, and we may thus assume that $w$ has finite order. We will prove that there is some $w'\in W$ with $w\sim_\kappa w'$ and such that $w'$ is SCR in $W$, yielding the claim. 

By Lemma~\ref{lemma geom combin CCCC} together with Proposition~\ref{prop CCCC cap Min}, one can find some $w_1\in W$ with $w\sim_\kappa w_1$ such that $w_1$ fixes a face of $C_0$, that is, $w_1\in W_I$ for some spherical subset $I$ of $S$.
It then follows from \cite[Theorem~2.6]{GeckKimPfeiffer} that there is some $w'\in W_I$ which is SCR in $W_I$ and such that $w_1\sim_\kappa w'$. In particular, it follows from \cite[Proposition~4.2]{MR2585575} that the parabolic closure of $w'$ in $W_I$ (or equivalently, in $W$) is standard. We then deduce from Lemma~\ref{lemma cyclically reduced inside Pc} that there is some $v\in \Pc(w')\leq W_I$ such that $v\inv wv$ is SCR in $W$. But as $w'$ is SCR in $W_I$, we know that $\ell(w')=\ell(v\inv wv)$, and hence $w'$ is in fact SCR in $W$, as desired.
\hspace{\fill}$\Box$

\section{Straight elements in Coxeter groups}\label{section proof of thm}
This section is devoted to the proof of Theorem~\ref{mainthm intro} and Corollaries~\ref{corintro CFC} and \ref{corintro Coxeter element}. Throughout this section, we fix a Coxeter system $(W,S)$.

\begin{lemma}\label{lemma log implies}
Let $w\in W$ be straight. Then $w$ is cyclically reduced (more precisely, $\ell(w)=\min\{\ell(v wv\inv) \ | \ v\in  W\}$) and torsion-free.
\end{lemma}
\begin{proof}
Let $w\in W$ be straight. Assume first for a contradiction that it is not cyclically reduced: $\ell(w)\geq\ell(vwv\inv)+1$ for some $v\in W$. Then for all $n\in\NN$, 
$$\ell(vw^nv\inv)+2\ell(v)\geq \ell(w^n)=n\ell(w)\geq n\ell(vwv\inv)+n\geq \ell(vw^nv\inv)+n,$$ a contradiction. 

Assume next for a contradiction that $w$ is not torsion-free: there is a reduced decomposition $w=w_In_I$ for some spherical subset $I\subseteq S$, some $w_I\in W_I\setminus\{1\}$ and some $n_I$ normalising $W_I$. Then for each $k\in\NN$ there is some $w_k\in W_I$ such that $w^k=w_kn_I^k$. As $W_I$ is finite, $w_k=w_l$ for some $k< l$ and hence $w^K=n_I^K$ for $K=l-k\in\NN^*$. Thus $K\ell(w)=\ell(w^K)=\ell(n_I^K)\leq K\ell(n_I)<K\ell(w)$, a contradiction. 
\end{proof}

%\begin{remark}\label{remark wN1!=nIN1!}
%Let $N_1$ denote the least common multiple of the orders of the finite (standard) parabolic subgroups of $W$.
%Note then that the proof of Lemma~\ref{lemma log implies} yields in particular that if $w\in W$ normalises some finite parabolic subgroup $W_I$, and if $w=w_In_I$ is the corresponding decomposition provided by Lemma~\ref{lemma normaliser}, then $w^{N_1!}=n_I^{N_1!}$.
%\end{remark}

\begin{lemma}\label{lemma log iff conj log}
Let $v,w\in W$ be such that $\ell(w)=\ell(vwv\inv)$. Then $w$ is straight if and only if $vwv\inv$ is straight.
\end{lemma}
\begin{proof}
Suppose that $w$ is straight and assume for a contradiction that $vwv\inv$ is not, that is, there exists some $K\in\NN$ such that $\ell((vwv\inv)^K)<K\ell(vwv\inv)$. As $\ell(w)=\ell(vwv\inv)$, it follows that $K\ell(w)-1\geq \ell(vw^Kv\inv)$. Hence $$n(K\ell(w)-1)\geq n\ell(vw^Kv\inv)\geq\ell(vw^{Kn}v\inv)\geq \ell(w^{Kn})-2\ell(v)=Kn\ell(w)-2\ell(v)$$ for all $n\in\NN$, yielding the desired contradiction for $n>2\ell(v)$.
\end{proof}

\begin{lemma}\label{lemma log if ess}
Let $w$ be an infinite order element of $W$ possessing an axis through $C_0$ (that is, such that $\Min(w)\cap C_0\neq\emptyset$). If the walls separating $C_0$ from $wC_0$ are all $w$-essential, then $w$ is straight.
\end{lemma}
\begin{proof}
Let $D$ be a $w$-axis through $C_0$. Assume for a contradiction that $w$ is not straight. Then there exists a wall $m$ separating $C_0$ from $wC_0$ and $w^nC_0$ from $w^{n+1}C_0$ for some nonzero $n\in\NN$. In particular, $m$ intersects $D$ in at least two points, hence contains $D$. But this contradicts the assumption that $m$ is $w$-essential.
\end{proof}

\begin{remark}\label{remark int C0}
Note that if $w$ is an infinite order element of $W$ possessing an axis through the interior $\interieur(C_0)$ of $C_0$, then it is straight as it satisfies the hypotheses of Lemma~\ref{lemma log if ess}: indeed, in that case, the walls separating $C_0$ from $wC_0$ are also the walls separating a point $x\in\Min(w)\cap \interieur(C_0)$ from $wx$ and such an $x$ is not contained in any wall.
\end{remark}

\begin{lemma}\label{lemma straight iff tf}
Let $w\in W$ be of infinite order. Assume that $\Min(w)\cap C_0$ is nonempty and not contained in any $w$-essential wall. Then $w$ is straight if and only if it is torsion-free.
\end{lemma}
\begin{proof}
Pick $x\in \Min(w)\cap C_0$ such that $x$ does not lie on any $w$-essential wall, and let $L$ be the $w$-axis through $x$. Let $A\in\Sigma$ be the support of $x$, say $A=W_I$ for some spherical subset $I\subseteq S$. Then the walls containing $x$ (that is, the walls containing $A$, or else the walls with associated roots in $\Phi_I$) also contain $L$. 
It follows that $w$ stabilises $\Phi_I$ and hence normalises $W_I$. Consider the decomposition $N_W(W_I)=W_I\rtimes N_I$ provided by Lemma~\ref{lemma normaliser} and write $w=w_In_I$ accordingly, with $w_I\in W_I$ and $n_I\in N_I$.

Notice that $L$ is also an axis for $n_I$ as $n_Ix=n_I(n_I\inv w_In_I)x=w_In_Ix=wx\in L$ and $n_I\inv x=n_I\inv w_I\inv x=w\inv x\in L$. Moreover, no wall containing $L$ separates $C_0$ from $n_IC_0$. Indeed, as $C_0$ is contained in the intersection of all roots $\alpha_s$ with $s\in I$, so is $n_IC_0$, yielding the claim. Hence $n_I$ is straight by Lemma~\ref{lemma log if ess}.

Now, if $w$ is torsion-free, then $w_I=1$ and $w=n_I$ is straight. The converse was established in Lemma~\ref{lemma log implies}.
\end{proof}

\begin{lemma}\label{lemma CFC help}
Let $u\in W$ be cyclically reduced. Then there is some $w\in W$ with $u\sim_{\kappa} w$ such that $u$ is straight if and only if $w$ is torsion-free.
\end{lemma}
\begin{proof}
Let $w$ be as in Corollary~\ref{corollary sim to C0}, so that in particular $\ell(u)=\ell(w)$. As $u$ is straight if and only if $w$ is straight by Lemma~\ref{lemma log iff conj log}, the claim readily follows from Lemma~\ref{lemma straight iff tf}.
\end{proof}

\noindent
{\bf Proof of Theorem~\ref{mainthm intro}.}
Let $u\in W$ be cyclically reduced. The implication (1) $\Rightarrow$ (2) follows from Lemma~\ref{lemma log iff conj log} together with Lemma~\ref{lemma log implies}. Assume now that (2) holds. Let $w\in W$ with $u\sim_{\kappa} w$ be as in Lemma~\ref{lemma CFC help}. Then $w$ is torsion-free by assumption, and hence $u$ is straight, as desired.
\hspace{\fill}$\Box$

\medskip

To prove Corollary~\ref{corintro CFC}, we need one more technical lemma.

\begin{lemma}\label{lemma FC tf}
Let $w\in W$ be an FC element. Assume that $w$ is not torsion-free. Then the standard parabolic closure of $w$ possesses a nontrivial spherical irreducible component.
\end{lemma}
\begin{proof}
Write $w=w_In_I$ for some spherical $I\subseteq S$, some $w_I\in W_I\setminus\{1_W\}$ and some $n_I\in W$ normalising $W_I$. Let $S_1$ (resp. $S_2$) denote the set of generators of $S$ appearing in a reduced decomposition of $n_I$ (resp. $w_I$). Thus the standard parabolic closure of $w$ coincides with $W_{S_1\cup S_2}$.

Note that if $\ell(w_I)=1$, say $w_I=s\in I$, then as $sn_I=n_Is'$ for some $s'\in I$, the FC condition implies that $s=s'$ and that $s$ commutes with every generator of $n_I$. Moreover, given an FC element $u$ and a reduced decomposition of the form $u=u_1u_2u_3$, the element $u_2$ is still FC. An easy induction on $\ell(w_I)$ thus yields that every generator in $S_1$ commutes with every generator in $S_2$. 

Let now $s\in S_2$ and let $W_T$ be the irreducible component of $W_{S_1\cup S_2}$ containing $s$. Write $T_i=T\cap S_i$ for $i=1,2$, so that $T=T_1\cup T_2$. Set also $T_1'=T_1\setminus (T_1\cap T_2)$. Then $T$ is the disjoint union of $T_1'$ and $T_2$, and every generator of $T_1'$ commutes with every generator of $T_2$. As $W_T$ is irreducible and $T_2$ contains $s$, we deduce that $T_1'=\emptyset$, that is, $T_1\subseteq T_2$. Hence $T=T_2\subseteq I$ and $W_T\subseteq W_I$ is spherical, as desired.
\end{proof}

\noindent
{\bf Proof of Corollary~\ref{corintro CFC}.}
The implication (1) $\Rightarrow$ (2) follows from Lemma~\ref{lemma log implies}. Let now $u$ be a CFC element whose standard parabolic closure has only infinite irreducible components. Note that the CFC condition implies that $u$ is cyclically reduced. Let $w$ be the conjugate of $u$ provided by Lemma~\ref{lemma CFC help}, so that in particular the standard parabolic closures of $w$ and $u$ coincide. As $w$ is FC (since every cyclic shift of a CFC element is again CFC, see e.g. \cite[Proposition~4.1]{CFCCoxeter}), it follows from Lemma~\ref{lemma FC tf} that it is torsion-free. Thus $u$ is straight by Lemma~\ref{lemma CFC help}, as desired. \hspace{\fill}$\Box$

\medskip

\noindent
{\bf Proof of Corollary~\ref{corintro Coxeter element}.}
As Coxeter elements are CFC elements, this readily follows from Corollary~\ref{corintro CFC}. \hspace{\fill}$\Box$

\bibliographystyle{amsalpha} 
\bibliography{these} 

\newcommand{\etalchar}[1]{$^{#1}$}
\def\cprime{$'$}
\providecommand{\bysame}{\leavevmode\hbox to3em{\hrulefill}\thinspace}
\providecommand{\MR}{\relax\ifhmode\unskip\space\fi MR }
% \MRhref is called by the amsart/book/proc definition of \MR.
\providecommand{\MRhref}[2]{%
  \href{http://www.ams.org/mathscinet-getitem?mr=#1}{#2}
}
\providecommand{\href}[2]{#2}
\begin{thebibliography}{BBE{\etalchar{+}}12}

\bibitem[AB08]{BrownAbr}
Peter Abramenko and Kenneth~S. Brown, \emph{Buildings}, Graduate Texts in
  Mathematics, vol. 248, Springer, New York, 2008, Theory and applications.

\bibitem[BBE{\etalchar{+}}12]{CFCCoxeter}
T.~Boothby, J.~Burkert, M.~Eichwald, D.~C. Ernst, R.~M. Green, and M.~Macauley,
  \emph{On the cyclically fully commutative elements of {C}oxeter groups}, J.
  Algebraic Combin. \textbf{36} (2012), no.~1, 123--148.

\bibitem[BH99]{BHCAT0}
Martin~R. Bridson and Andr{\'e} Haefliger, \emph{Metric spaces of non-positive
  curvature}, Grundlehren der Mathematischen Wissenschaften [Fundamental
  Principles of Mathematical Sciences], vol. 319, Springer-Verlag, Berlin,
  1999.

\bibitem[Bri99]{Bridson}
Martin~R. Bridson, \emph{On the semisimplicity of polyhedral isometries}, Proc.
  Amer. Math. Soc. \textbf{127} (1999), no.~7, 2143--2146.

\bibitem[CF10]{MR2585575}
Pierre-Emmanuel Caprace and Koji Fujiwara, \emph{Rank-one isometries of
  buildings and quasi-morphisms of {K}ac--{M}oody groups}, Geom. Funct. Anal.
  \textbf{19} (2010), no.~5, 1296--1319.

\bibitem[DJ99]{DraJa}
A.~Dranishnikov and T.~Januszkiewicz, \emph{Every {C}oxeter group acts amenably
  on a compact space}, Proceedings of the 1999 {T}opology and {D}ynamics
  {C}onference ({S}alt {L}ake {C}ity, {UT}), vol.~24, 1999, pp.~135--141.

\bibitem[EE09]{MR2515767}
Henrik Eriksson and Kimmo Eriksson, \emph{Conjugacy of {C}oxeter elements},
  Electron. J. Combin. \textbf{16} (2009), no.~2, Special volume in honor of
  Anders Bjorner, Research Paper 4, 7.

\bibitem[GKP00]{GeckKimPfeiffer}
Meinolf Geck, Sungsoon Kim, and G{\"o}tz Pfeiffer, \emph{Minimal length
  elements in twisted conjugacy classes of finite {C}oxeter groups}, J. Algebra
  \textbf{229} (2000), no.~2, 570--600.

\bibitem[Kra09]{Krammer}
Daan Krammer, \emph{The conjugacy problem for {C}oxeter groups}, Groups Geom.
  Dyn. \textbf{3} (2009), no.~1, 71--171.

\bibitem[Lus77]{Lusztig}
George Lusztig, \emph{Coxeter orbits and eigenspaces of {F}robenius}, Invent.
  Math. \textbf{38} (1976/77), no.~2, 101--159.

\bibitem[Mat64]{Matsumoto}
Hideya Matsumoto, \emph{G\'en\'erateurs et relations des groupes de {W}eyl
  g\'en\'eralis\'es}, C. R. Acad. Sci. Paris \textbf{258} (1964), 3419--3422.

\bibitem[NV02]{singlepoint}
Guennadi~A. Noskov and {\`E}rnest~B. Vinberg, \emph{Strong {T}its alternative
  for subgroups of {C}oxeter groups}, J. Lie Theory \textbf{12} (2002), no.~1,
  259--264.

\bibitem[Spe09]{wkreduced}
David~E. Speyer, \emph{Powers of {C}oxeter elements in infinite groups are
  reduced}, Proc. Amer. Math. Soc. \textbf{137} (2009), no.~4, 1295--1302.

\bibitem[Tit69]{Titswordproblem}
Jacques Tits, \emph{Le probl\`eme des mots dans les groupes de {C}oxeter},
  Symposia {M}athematica ({INDAM}, {R}ome, 1967/68), {V}ol. 1, Academic Press,
  London, 1969, pp.~175--185.

\end{thebibliography}

\end{document}